\documentclass[12pt,a4paper, reqno]{article}
\usepackage[utf8]{inputenc}
\usepackage[T1]{fontenc}
\usepackage{float}
\usepackage{fullpage}
\usepackage[dvipsnames]{xcolor}
\usepackage[english]{babel}
\usepackage{calc}
\usepackage{amssymb}
\usepackage{enumitem}
\usepackage{booktabs}
\usepackage{amsmath}
\usepackage{amsthm}
\usepackage{colortbl}
\usepackage{longtable}
\usepackage{graphicx}
\usepackage{xurl}
\usepackage{array}
\usepackage{booktabs}
\usepackage{multirow}
\usepackage[labelformat=simple]{subcaption}

\usepackage[export]{adjustbox}
\definecolor{webgreen}{rgb}{0,.5,0}
\definecolor{webbrown}{rgb}{.6,0,0}
\usepackage[colorlinks=true,
linkcolor=webgreen,
filecolor=webbrown,
citecolor=webgreen]{hyperref}
\usepackage{tikz}
\usepackage{bm}

\newtheorem{theorem}{Theorem}[section]
\newtheorem{notation}[theorem]{Notation}
\newtheorem{lemma}[theorem]{Lemma}
\theoremstyle{definition}
\newtheorem{definition}{Definition}[section]
\theoremstyle{remark}
\newtheorem{remark}{Remark}

\newcommand{\seqnum}[1]{\href{http://oeis.org/#1}{\underline{#1}}}
\everymath{\displaystyle}

\newcommand{\CK}{\mathit{CK}}
\newcommand{\go}{\protect\includegraphics[height=1.75cm,valign=c]{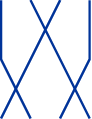}}
\newcommand{\gi}{\protect\includegraphics[height=1.75cm,valign=c]{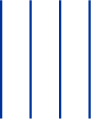}}
\newcommand{\gii}{\protect\includegraphics[height=1.75cm,valign=c]{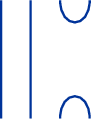}}
\newcommand{\giii}{\protect\includegraphics[height=1.75cm,valign=c]{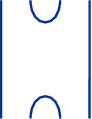}}
\newcommand{\giv}{\protect\includegraphics[height=1.75cm,valign=c]{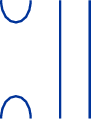}}
\newcommand{\gv}{\protect\includegraphics[height=1.75cm,valign=c]{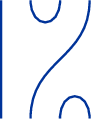}}
\newcommand{\gvi}{\protect\includegraphics[height=1.75cm,valign=c]{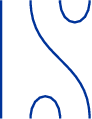}}
\newcommand{\gvii}{\protect\includegraphics[height=1.75cm,valign=c]{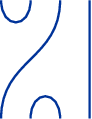}}
\newcommand{\gviii}{\protect\includegraphics[height=1.75cm,valign=c]{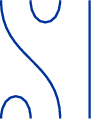}}
\newcommand{\gix}{\protect\includegraphics[height=1.75cm,valign=c]{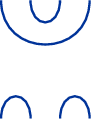}}
\newcommand{\gx}{\protect\includegraphics[height=1.75cm,valign=c]{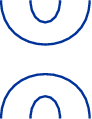}}
\newcommand{\gxi}{\protect\includegraphics[height=1.75cm,valign=c]{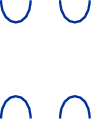}}
\newcommand{\gxii}{\protect\includegraphics[height=1.75cm,valign=c]{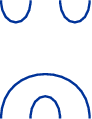}}
\newcommand{\gxiii}{\protect\includegraphics[height=1.75cm,valign=c]{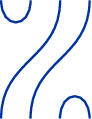}}
\newcommand{\gxiv}{\protect\includegraphics[height=1.75cm,valign=c]{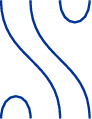}}
\newcommand{\gicI}{\protect\includegraphics[height=2.5cm,valign=c]{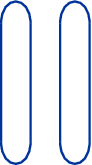}}
\newcommand{\giicI}{\protect\includegraphics[height=2.5cm,valign=c]{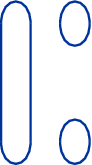}}
\newcommand{\giiicI}{\protect\includegraphics[height=2.5cm,valign=c]{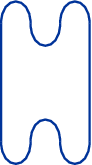}}
\newcommand{\givcI}{\protect\includegraphics[height=2.5cm,valign=c]{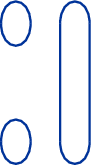}}
\newcommand{\gvcI}{\protect\includegraphics[height=2.5cm,valign=c]{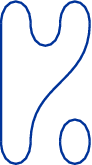}}
\newcommand{\gvicI}{\protect\includegraphics[height=2.5cm,valign=c]{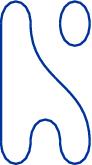}}
\newcommand{\gviicI}{\protect\includegraphics[height=2.5cm,valign=c]{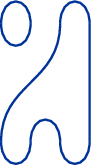}}
\newcommand{\gviiicI}{\protect\includegraphics[height=2.5cm,valign=c]{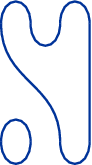}}
\newcommand{\gixcI}{\protect\includegraphics[height=2.5cm,valign=c]{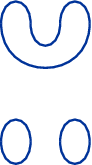}}
\newcommand{\gxcI}{\protect\includegraphics[height=2.5cm,valign=c]{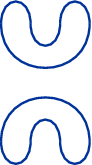}}
\newcommand{\gxicI}{\protect\includegraphics[height=2.5cm,valign=c]{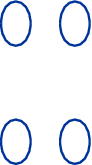}}
\newcommand{\gxiicI}{\protect\includegraphics[height=2.5cm,valign=c]{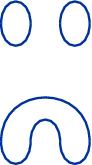}}
\newcommand{\gxiiicI}{\protect\includegraphics[height=2.5cm,valign=c]{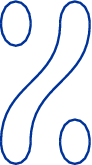}}
\newcommand{\gxivcI}{\protect\includegraphics[height=2.5cm,valign=c]{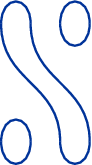}}

\numberwithin{equation}{section}

\makeatletter
\def\namedlabel#1#2{\begingroup
	#2%
	\def\@currentlabel{#2}%
	\phantomsection\label{#1}\endgroup
}
\makeatother

\title{\bf The bracket polynomial of the   Celtic link shadow $\bm{\CK_4^{2n}}$}
\author{Franck Ramaharo\\
	\small Mention Mathématiques et Informatique\\[-0.8ex]
	\small Université d'Antananarivo\\[-0.8ex] 
	\small 101 Antananarivo, Madagascar\\
	\small\href{mailto:franck.ramaharo@gmail.com}{\tt franck.ramaharo@gmail.com}\\
}
\date{\small August 14, 2025\\}

\begin{document}
	\maketitle

	\begin{abstract}
		We derive the Kauffman bracket polynomial for the shadow of the Celtic link $\CK_4^{2n}$ using two complementary approaches. The first approach uses a recursive relation within the Celtic framework of Gross and Tucker, based on diagrammatic identities. The second approach makes use of a 4-tangle algebraic framework: a fundamental tangle is concatenated with itself $n$ times to form an iterated composite tangle, and the Kauffman bracket polynomial is computed by decomposing the state space with respect to the basis elements of the 4-strand diagram monoid.

	\end{abstract}
	
	\section{Introduction}
	
	The aim of this paper is to compute the Kauffman bracket polynomial of the Celtic link shadow $\CK_4^{2n}$. An example of such a shadow diagram is shown in Figure~\ref{fig:CK420}. 
	
	\begin{figure}[H]
		\centering
		\centering
		\includegraphics[height=1.5cm]{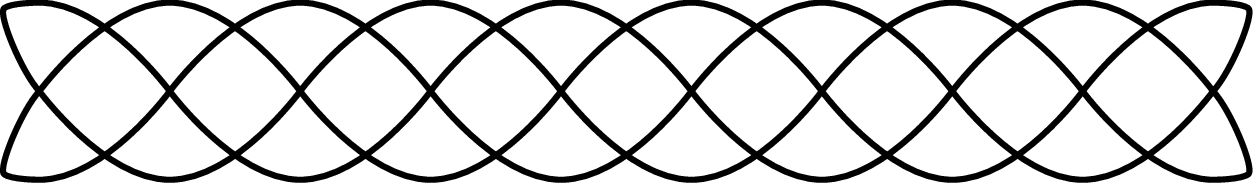}
		\caption{Celtic knot shadow $\CK^{20}_4$}
		\label{fig:CK420}
	\end{figure}
	
	In this paper, we work exclusively with the shadow diagram \cite[p.~11--12]{Kauffman1993}, that is the 4-regular plane graph underlying the link projection with no information about over- or under-crossings. For simplicity, the term ``knot'' will be used throughout this paper to refer to either a knot or a link, and specifically to its shadow diagram. Moreover, all diagrams are considered up to isotopy on the sphere $\mathcal{S}^2$. 
	
	We present two complementary approaches to compute the Kauffman bracket polynomial of the Celtic knot $\CK_4^{2n}$. The first approach uses a recursive relation within the Celtic framework of Gross and Tucker \cite{GrossTucker2011} by applying diagrammatic identities.  Since we work with the shadow diagram, the recursive relation simplifies significantly and allows us to derive an explicit formula for $\left<\CK_4^{2n}\right>$.
	
	The second approach makes use of a 4-tangle algebraic framework. A fundamental tangle is concatenated with itself $n$ times to form an iterated composite tangle. The Kauffman bracket polynomial is then computed via a recursive process. The strategy consists of decomposing the states of the fundamental tangle with respect to a basis of 14 elements from the 4-strand diagram monoid $\mathcal{K}_4$ \cite{KitovNikov2020}. 
	
	The remainder of the paper is organized as follows. In Section~\ref{sec:bracket}, we recall the definition of the bracket polynomial for a knot shadow. In Section~\ref{sec:celticframework}, we compute the bracket polynomial of the Celtic knot $\CK_4^{2n}$  within the Celtic framework. Finally, in Section~\ref{sec:tangle framework}, we compute the bracket polynomial of $\CK_4^{2n}$  using the 4-tangle algebraic framework.

	\section{Kauffman bracket polynomials}\label{sec:bracket}
	
	The Kauffman bracket polynomial for a knot assigns to each shadow diagram $ D $ an element $ \langle D \rangle \in \mathbb{Z}[x] $. It is defined recursively by the following three axioms:
	
	\begin{itemize}
		\item [\namedlabel{itm:K1}{$ (\mathbf{K1}) $}:] $  \left<\bigcirc \right>=x $;
		\item [\namedlabel{itm:K2}{$ (\mathbf{K2}) $}:]  $ \left<\bigcirc\sqcup D'\right>=x\left< D'\right>$;
		\item [\namedlabel{itm:K3}{$ (\mathbf{K3}) $}:]  $\left<\protect\includegraphics[width=.0225\linewidth,valign=c]{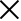}\right>=\left<\protect\includegraphics[width=.0225\linewidth,valign=c]{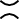}\right>+\left<\protect\includegraphics[width=.0225\linewidth,valign=c]{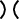}\right>$.
	\end{itemize}
	
	Here, $ \bigcirc $ denotes a simple closed curves or a circle, $ \sqcup $ denotes disjoint union, and the local diagrams in \ref{itm:K3} represent the two possible smoothings of a crossing. The latter are called the states of that crossing. A state of the diagram refers to the complete configuration obtained by choosing a smoothing at every crossing, resulting in a collection of disjoint circles. 
	
	In this setting, the bracket polynomial has a purely combinatorial interpretation. Specifically, $ \langle D \rangle $ is a state-sum over all possible ways of smoothing every crossing in $ D $:	\[
	\langle D \rangle = \sum_{S} x^{|S|},
	\] 	where the sum is taken over all Kauffman states $ S $ of $ D $, and $ |S| $ denotes the number of disjoint circles in the state $ S $. Thus, the coefficient $ \left[x^k\right] \langle D \rangle $ represents the number of states in which the resolved diagram consists of exactly $ k $ circles.

	\section{Within the Celtic framework}\label{sec:celticframework}
	
	Following the Celtic framework of Gross and Tucker \cite{GrossTucker2011}, the knot $\CK_4^{2n}$ is the barrier-free Celtic knot constructed on a $4\times 2n$  rectangular grid of squares, using the standard Celtic design rules~\cite{Bain1992,Sloss2002}.  The Celtic knots $\CK_4^{2}$, $\CK_4^{4}$, $\CK_4^{6}$ and $\CK_4^{2n}$ are shown in Figure~\ref{fig:Celticdefinition}.
	\begin{figure}[H]
		\centering
		\begin{subfigure}[t]{0.2\textwidth}
			\centering
			\includegraphics[height=3.05cm]{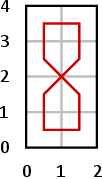}
			\caption{$\CK_4^{2}$}
		\end{subfigure}%
		\hfill 
		\begin{subfigure}[t]{0.2\textwidth}
			\centering
			\includegraphics[height=3.05cm]{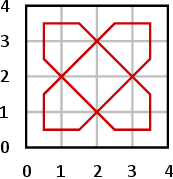}
			\caption{$\CK_4^{4}$}
		\end{subfigure}%
		\hfill	
		\begin{subfigure}[t]{0.3\textwidth}
			\centering
			\includegraphics[height=3.05cm]{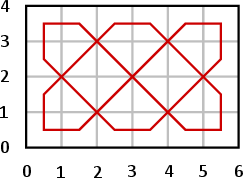}
			\caption{$\CK_4^{6}$}
		\end{subfigure}%
		\hfill 
		\begin{subfigure}[t]{0.3\textwidth}
			\centering
			\includegraphics[height=3.05cm]{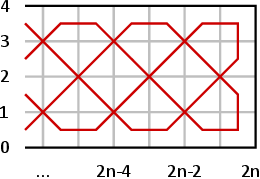}
			\caption{$\CK_4^{2n}$}
		\end{subfigure}%
		\caption{Celtic knot diagrams}
		\label{fig:Celticdefinition}
	\end{figure}
	
	\begin{remark}
		In the Celtic framework, axiom \ref{itm:K3} can be written as 
		\begin{equation}\label{eq:K3p}
			\left<\CK^{2n}_4\right> = \left<H_i^j\CK^{2n}_4\right>+ \left<V_i^j\CK^{2n}_4\right>,
		\end{equation} where the notation $H_i^j$ means replace the crossing of a Celtic knot at row $i$, column $j$ by a horizontal pair. The notation $V_i^j$ means replace the crossing at row $i$, column $j$ by a vertical pair \cite{GrossTucker2011}. See example in Figure~\ref{fig:bracketoperation}.
		
		\begin{figure}[!h]
			\centering
			\begin{tikzpicture}
				\node at (0,0) {\includegraphics[width=0.58\linewidth]{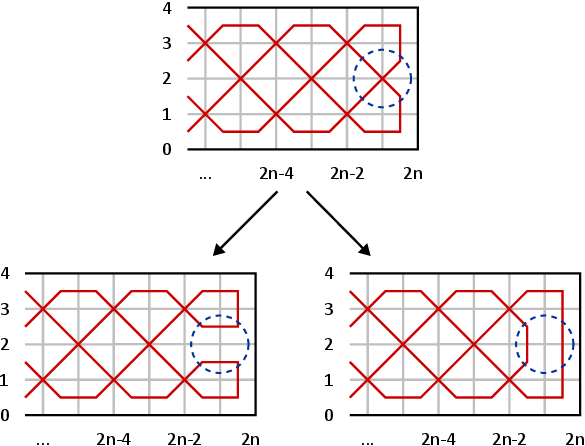}};
				\node at (0,4) {$\CK^{2n}_4$};
				\node at (-2.5,-0.5) {$H_2^{2n-1}\CK^{2n}_4$};
				\node at (2.5,-0.5) {$V_2^{2n-1}\CK^{2n}_4$};
			\end{tikzpicture}
			\caption{Illustration of $\left<\CK^{2n}_4\right> = \left<H_2^{2n-1}\CK^{2n}_4\right>+ \left<V_2^{2n-1}\CK^{2n}_4\right>$}
			\label{fig:bracketoperation}
		\end{figure}
	\end{remark}
	
	Having introduced the necessary notation, we derive the following Lemma and Theorem from the recursive relations for the Kauffman bracket polynomial given by Gross and Tucker~\cite{GrossTucker2011}.
	
	\begin{lemma}
		The following four relations hold for bracket polynomials:
		\begin{align}
			\left<H^{2n-1}_2\CK_4^{2n}\right>&=(x+1)^2\left<\CK_4^{2n-2}\right>; \label{eq:R1}\\
			\left<V_1^{2n-2}V_2^{2n-1}\CK_4^{2n}\right>&=(x+1)\left<\CK_4^{2n-2}\right>; \label{eq:R2}\\
			\left<H_3^{2n-2}H_1^{2n-2}V_2^{2n-1}\CK_4^{2n}\right>&=(x+1)\left<V_2^{2n-3}\CK_4^{2n-2}\right>;\label{eq:R3}\\
			\left<V_3^{2n-2}H_1^{2n-2}V_2^{2n-1}\CK_4^{2n}\right>&=\left<\CK_4^{2n-2}\right>. \label{eq:R4}
		\end{align}
	\end{lemma}
	
	\begin{proof}
		Formulas \eqref{eq:R1}--\eqref{eq:R4} follow directly from the diagrammatic relations depicted in Figure~\ref{subfig:bracketHCK}--\subref{subfig:bracketVHVCK}, respectively.
		
		\begin{figure}[!h]
			\centering			
			\begin{subfigure}[t]{0.4\textwidth}
				\centering
				\includegraphics[height=3.05cm]{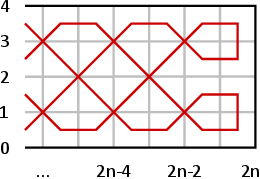}
				\caption{$H^{2n-1}_2\CK_4^{2n}$}
				\label{subfig:bracketHCK}
			\end{subfigure}%
			~
			\begin{subfigure}[t]{0.4\textwidth}
				\centering
				\includegraphics[height=3.05cm]{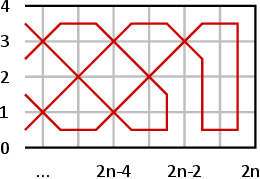}
				\caption{$V_1^{2n-2}V_2^{2n-1}\CK_4^{2n}$}
			\end{subfigure}
			
			\begin{subfigure}[t]{0.4\textwidth}
				\centering
				\includegraphics[height=3.05cm]{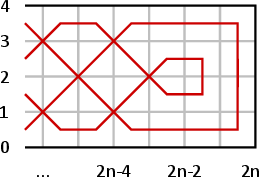}
				\caption{$H_3^{2n-2}H_1^{2n-2}V_2^{2n-1}\CK_4^{2n}$}
			\end{subfigure}%
			~
			\begin{subfigure}[t]{0.4\textwidth}
				\centering
				\includegraphics[height=3.05cm]{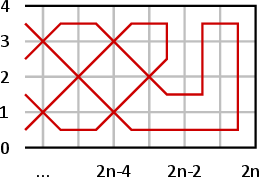}
				\caption{$V_3^{2n-2}H_1^{2n-2}V_2^{2n-1}\CK_4^{2n}$}
				\label{subfig:bracketVHVCK}
			\end{subfigure}
			\caption{Celtic shadow diagrams for bracket polynomial relations}
		\end{figure}	
	\end{proof}
	
	\begin{lemma}
		The bracket polynomial for the sequence $\left(\CK^{2n}_4\right)_n$ is given by the following recursion
		\begin{align}
			\left<\CK^2_4\right>&=x^2+x;\label{eq:Rec2}\\
			\left<V_2^1\CK^2_4\right>&=x;\label{eq:Rec3}\\
			\left<\CK^{2n}_4\right>&=(x+1)^2\left<\CK_4^{2n-2}\right>+\left<V_2^{2n-1}\CK_4^{2n}\right>\quad \textit{for $n\geq 2$}\label{eq:Rec4};\\
			\left<V_2^{2n-1}\CK^{2n}_4\right>&=(x+1)\left<V_2^{2n-3}\CK_4^{2n-2}\right>+(x+2)\left<\CK^{2n-2}_4\right>\quad \textit{for $n\geq 2$}.\label{eq:Rec5}
		\end{align}
	\end{lemma}
	
	\begin{proof}
		Formulas \eqref{eq:Rec2} and \eqref{eq:Rec3} are easily verified. For \eqref{eq:Rec4}, we have
		\begin{align*}
			\left<\CK^{2n}_4\right>&=\left<H_2^{2n-1}\CK_4^{2n}\right>+\left<V_2^{2n-1}\CK_4^{2n}\right>\quad \mbox{(by axiom \eqref{eq:K3p})}\\
			&=(x+1)^2\left<\CK_4^{2n-2}\right>+\left<V_2^{2n-1}\CK_4^{2n}\right> \quad \mbox{(by \eqref{eq:R1})}.
		\end{align*}
		For \eqref{eq:Rec5}, we have
		\begin{align*}
			\left<V_2^{2n-1}\CK^{2n}_4\right>&=	\left<H_1^{2n-2}V_2^{2n-1}\CK^{2n}_4\right>+	\left<V_1^{2n-2}V_2^{2n-1}\CK^{2n}_4\right>\quad \mbox{(by axiom \eqref{eq:K3p})}\\
			&= \left<H_1^{2n-2}V_2^{2n-1}\CK^{2n}_4\right>+(x+1)\left<\CK_4^{2n-2}\right>  \quad \mbox{(by \eqref{eq:R2})}\\
			&= \left<H_3^{2n-2}H_1^{2n-2}V_2^{2n-1}\CK^{2n}_4\right>+\left<V_3^{2n-2}H_1^{2n-2}V_2^{2n-1}\CK^{2n}_4\right>\quad \mbox{(by axiom \eqref{eq:K3p})}\\
			&\qquad+(x+1)\left<\CK_4^{2n-2}\right>\\
			&= (x+1)\left<V_2^{2n-3}\CK_4^{2n-2}\right>+\left<\CK_4^{2n-2}\right>+(x+1)\left<\CK_4^{2n-2}\right> \quad \mbox{(by \eqref{eq:R3} and \eqref{eq:R4})}\\
			&= (x+1)\left<V_2^{2n-3}\CK_4^{2n-2}\right>+(x+2)\left<\CK^{2n-2}_4\right>.
		\end{align*}
	\end{proof}
	
	These recursions enable us to derive a closed-form expression for $\left<\CK^{2n}_4\right>$. 
	
	\begin{theorem}
		The bracket polynomial of the Celtic knot $\CK_4^{2n}$ is given by 
		\begin{equation}\label{eq:Celticclosedform}
			\left<\CK^{2n}_4\right> = \dfrac{1}{2q}\left(x^2+x\right)\left(\left(x^2+2x+4+q\right)\left(\dfrac{p+q}{2}\right)^{n-1}-\left(x^2+2x+4-q\right)\left(\dfrac{p-q}{2}\right)^{n-1}\right),
		\end{equation}
		where \[p:=x^2+4x+4  \quad\text{and}\quad q:=\sqrt{x^4+4 x^3+12 x^2+20 x+12}.\]
	\end{theorem}
	
	\begin{proof}
		Formulas \eqref{eq:Rec4} and \eqref{eq:Rec5} can be expressed in the following matrix form:
		
		\begin{equation}\label{eq:matrixform}
			\begin{aligned}
				\begin{pmatrix}
					\left<\CK_4^{2n}\right>\\
					\left<V_2^{2n-1}\CK_4^{2n}\right>
				\end{pmatrix}&=\begin{pmatrix}
					x^2+3x+3 & x+1\\
					x+2 & x+1
				\end{pmatrix}\begin{pmatrix}
					\left<\CK_4^{2n-2}\right>\\
					\left<V_2^{2n-3}\CK_4^{2n-2}\right>
				\end{pmatrix}\\
				& =\begin{pmatrix}
					x^2+3x+3 & x+1\\
					x+2 & x+1
				\end{pmatrix}^{n-1}\begin{pmatrix}
					\left<\CK_4^{2}\right>\\
					\left<V_2^{1}\CK_4^{2}\right>
				\end{pmatrix}\\
				&=\begin{pmatrix}
					x^2+3x+3 & x+1\\
					x+2 & x+1
				\end{pmatrix}^{n-1}\begin{pmatrix}
					x^2+x\\
					x
				\end{pmatrix}\quad \mbox{by (\eqref{eq:Rec2} and \eqref{eq:Rec3})}.
			\end{aligned}
		\end{equation}
		
		Note that \eqref{eq:matrixform} remains valid for $n=1$. Let $M:=\begin{pmatrix}
			x^2+3x+3 & x+1\\
			x+2 & x+1
		\end{pmatrix}$ be referred to as the \textit{states matrix}. The characteristic polynomial of $M$ is given by
		\[
		\chi(M,\lambda)=  \left(\lambda-\dfrac{p-q}{2} \right)\left(\lambda-\dfrac{p+q}{2} \right), 
		\]
		where \[p:=x^2+4x+4  \quad\text{and}\quad q:=\sqrt{x^4+4 x^3+12 x^2+20 x+12}.\]	 The $n$-th power of the state matrix is computed using the standard eigenvalue method. We omit the details here, but the computation yields the following results:
		\begin{equation*}
			\left<\CK^{2n}_4\right> = \dfrac{1}{2q}\left(x^2+x\right)\left(\left(x^2+2x+4+q\right)\left(\dfrac{p+q}{2}\right)^{n-1}-\left(x^2+2x+4-q\right)\left(\dfrac{p-q}{2}\right)^{n-1}\right)
		\end{equation*} 
		and
		\begin{equation*}
			\left<V_2^{2n-1}\CK^{2n}_4\right> = \dfrac{1}{2q}x\left(\left(x^2+4x+2+q\right)\left(\dfrac{p+q}{2}\right)^{n-1}-\left(x^2+4x+2-q\right)\left(\dfrac{p-q}{2}\right)^{n-1}\right).
		\end{equation*}
		
	\end{proof}
	Let $\CK(x,y):=\sum_{n\geq 1}^{}\left<\CK^{2n}_4\right>y^n$ denote the generating function of $\left(\left<\CK^{2n}_4\right>\right)_n$. By \eqref{eq:Celticclosedform}, we have
	\begin{equation}\label{eq:gfpq}
		\CK(x,y)=\dfrac{2 xy (x+1) ((x^2 +2 x +4 -p)y+2)}{(2 - (p  - q) y) (2 - (p + q) y)}.
	\end{equation}
	
	Substituting the expressions for $p$  and $q$  in terms of $x$, \eqref{eq:gfpq} simplifies to 
	\begin{equation}
		\CK(x,y)=\dfrac{xy (x+1) (1-x y)}{1-(x+2)^2 y+(x+1)^3 y^2}.
	\end{equation}
	
	Table~\ref{tab:coefofCK} displays the coefficients of the Kauffman bracket polynomial for the Celtic knot $\CK^{2n}_4$ for small values of $n$. The distribution of $\left(\left[x^k\right]\left<\CK^{2n}_4\right>\right)_{n,k}$ is recorded as sequence \seqnum{A386874} in the OEIS \cite{Sloane2025}. Recall that the coefficient  $\left[x^k\right]\left<\CK^{2n}_4\right>$ represents the number of Kaufman states of  the Celtic knot $\CK^{2n}_4$ that consist of exactly $k$  circles.
	\begin{table}[!h]
		\centering
		\makebox[0pt]{\resizebox{1.07\linewidth}{!}{%
			\begin{tabular}{@{}c|lllllllllllllll@{}}
				$n\setminus k$ & 0 & 1 & 2 & 3 & 4 & 5 & 6 & 7 & 8 & 9 & 10 & 11 & 12  & 13 & 14\\	
				\midrule
				1 & 0 &  1    &  1 \\
				2 & 0 &  4    &  7     & 4     & 1\\
				3 & 0 &  15   &  40    & 42    & 23    & 7      & 1\\
				4 & 0 &  56   &  201   & 306   & 262   & 140    & 48     & 10    & 1\\
				5 & 0 &  209  &  943   & 1877  & 2189  & 1672   & 881    & 325   & 82    & 13    & 1\\
				6 & 0 &  780  &  4239  & 10412 & 15368 & 15276  & 10841  & 5660  & 2194  & 624   & 125  & 16   & 1\\
				7 & 0 &  2911 &  18506 & 54051 & 96501 & 118175 & 105495 & 71107 & 36885 & 14817 & 4579 & 1064 & 177 & 19 &  1\\
				& $\ldots$
		\end{tabular}}}
		\caption{Coefficients in the expansion of $\left<\CK^{2n}_4\right>$ for small values of $n$}
		\label{tab:coefofCK}
	\end{table}
	
	\begin{remark}\mbox{}
		\begin{itemize}
			\item Column 1 in Table~\ref{tab:coefofCK} matches sequence \seqnum{A001353} in the OEIS \cite{Sloane2025},
			which is related to the determinant sequence for the weaving knots $S(4, n, (1, -1, 1))$ \cite{Kimetal2016}.
			
			\item For $n=2$, $\CK_4^4$ is a ``$4$-foil'', the shadow of the link $4_1^2$. The bracket polynomial of an $n$-foil, denoted $F_n$, that is the shadow of the $(2,n)$-torus link \cite{Ramaharo2018}, is \begin{equation}\label{eq:nfoil}
				\left<F_n\right> = (x+1)^n+x^2-1.
			\end{equation}			
			\item For $n=3$, we have the following decomposition. 
			
			First,  we write $\left<\CK^{6}_4\right> = \left<H_2^{3}\CK^{6}_4\right>+ \left<V_2^{3}\CK^{6}_4\right>$ by axiom \eqref{eq:K3p} (see Figure~\ref{fig:bracketCK4_6}).
			\begin{figure}[H]
				\centering
				\begin{subfigure}[t]{0.4\textwidth}
					\centering
					\includegraphics[height=3.05cm]{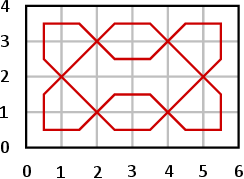}
					\caption{$H_2^{3}\CK^{6}_4$}
				\end{subfigure}%
				\begin{subfigure}[t]{0.4\textwidth}
					\centering
					\includegraphics[height=3.05cm]{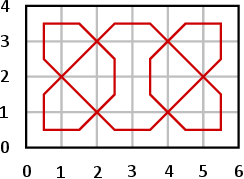}
					\caption{$V_2^{3}\CK^{6}_4$}
				\end{subfigure}%
				\caption{The states of the crossing at the construction dot in line 2 and column 3}
				\label{fig:bracketCK4_6}
			\end{figure}
		\end{itemize}
		Knot $H_2^{3}\CK^{6}_4$ is a $6$-foil (denoted $F_6$), while knot $V_2^{3}\CK^{6}_4$ is the connected sum of two trefoils (denoted $F_3\# F_3$). Recall that the bracket polynomial of the connected sum of two shadow diagrams $K$  and $K'$ satisfies \cite{Ramaharo2018}		
		\begin{equation}\label{eq:connectedsum}
			\left<K\# K'\right> = x^{-1}\left<K\right>\left<K'\right>.
		\end{equation}
		Hence, $\left<\CK^{6}_4\right>$ can be expressed as
		\begin{align*}
			\left<\CK^{6}_4\right> &= \left<F_6\right> + \left<F_3\# F_3\right>\\
			& = \left<F_6\right> + x^{-1}\left<F_3\right>^2\quad\mbox{(by \eqref{eq:connectedsum})}\\
			& = (x+1)^6 + x^2 - 1+ x^{-1}\left((x+1)^3+x^2-1\right)^2\quad\mbox{(by \eqref{eq:nfoil})}\\
			& = x^6+7 x^5+23 x^4+42 x^3+40 x^2+15 x.
		\end{align*}
	\end{remark}
	
	\begin{remark}
		In the context of Celtic design, the coefficient $\left[x^k\right]\left<\CK^{2n}_4\right>$ denotes the number of distinct barrier configurations \cite{GrossTucker2011} or equivalently, the number of break assignments \cite{AntonsenTaalman2021} that result in exactly $k$  connected components in the final design. Specifically, each such configuration corresponds to a choice of either a horizontal barrier (type-0 break) or a vertical barrier (type-1 break) at every interior construction dot (i.e., each crossing in the shadow diagram) which fully determine the state of the plaitwork.
		
		Let us consider the case $n=3$ and focus on the two extreme values $k=1$  and $k=6$. 
		
		For $k=1$ , there are 15 distinct barrier configurations that result in a single connected component. These 15 configurations fall into six equivalence classes under the action of the dihedral group $\mathcal{D}_2$, which includes reflection across the vertical axis, reflection across the horizontal axis and 180\textdegree~rotation \cite{AntonsenTaalman2021}. The representatives of these equivalence classes are shown in Figure~\ref{fig:equivalenceclass}. The size of equivalence classes are for Figure~\ref{subfig:s1}: 2,  \ref{subfig:s2}: 2, \ref{subfig:s3}: 2, \ref{subfig:s4}: 1, \ref{subfig:s5}: 4, and \ref{subfig:s2}: 4. 
		
		\begin{figure}[H]
			\centering
			
			\begin{subfigure}[t]{0.33\textwidth}
				\centering
				\includegraphics[height=3.05cm]{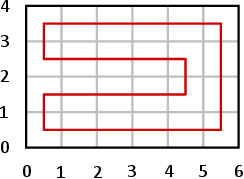}
				\caption{}
				\label{subfig:s1}
			\end{subfigure}%
			\hfill 
			\begin{subfigure}[t]{0.33\textwidth}
				\centering
				\includegraphics[height=3.05cm]{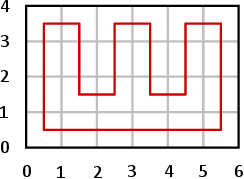}
				\caption{}
				\label{subfig:s2}
			\end{subfigure}
			\hfill 
			\begin{subfigure}[t]{0.33\textwidth}
				\centering
				\includegraphics[height=3.05cm]{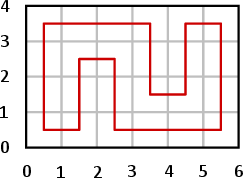}
				\caption{}
				\label{subfig:s3}
			\end{subfigure}
			
			\begin{subfigure}[t]{0.33\textwidth}
				\centering
				\includegraphics[height=3.05cm]{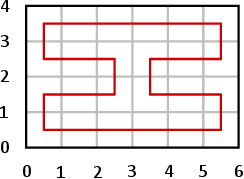}
				\caption{}
				\label{subfig:s4}
			\end{subfigure}%
			\hfill 
			\begin{subfigure}[t]{0.33\textwidth}
				\centering
				\includegraphics[height=3.05cm]{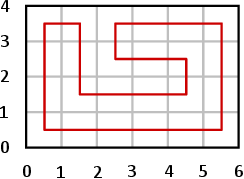}
				\caption{}
				\label{subfig:s5}
			\end{subfigure}
			\hfill 
			\begin{subfigure}[t]{0.33\textwidth}
				\centering
				\includegraphics[height=3.05cm]{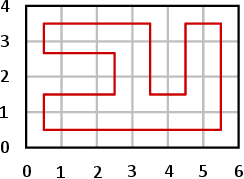}
				\caption{}
				\label{subfig:s6}
			\end{subfigure}
			\caption{The 6 equivalence classes for the single-component barrier configurations  on $\CK_4^6$}
			\label{fig:equivalenceclass}
		\end{figure}
		For $k=6$, there is only one such configuration that produces exactly 6 disjoint simple closed curves. This unique configuration is illustrated in Figure~\ref{fig:ck4611}. 
		\begin{figure}[H]
			\centering
			\includegraphics[height=3.05cm]{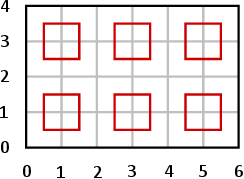}
			\caption{The unique 6-component barrier configuration for $ \CK_4^6 $}
			\label{fig:ck4611}
		\end{figure}		
	\end{remark}

	\section{Within the 4-tangle framework}\label{sec:tangle framework}
	In this section, we make use of a 4-tangle algebraic framework to analyze the Kauffman bracket of Celtic knot shadows.
	
	Kauffman's state model for tangles shows that any state $S$  of a $4$-tangle diagram $ T :=\protect\includegraphics[height=1.5cm,valign=c]{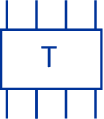}$  lies in the module generated by the disjoint union of loops and elements of the 4-strand diagram monoid $\mathcal{K}_4$ which consists of all planar isotopy classes of 4-endpoint diagrams with no closed components \cite{KitovNikov2020}. 
	
	The monoid $\mathcal{K}_4$   has 14 distinct basis elements, here denoted $g_1,g_2,\ldots,g_{14}$, corresponding to the 14 possible ways to connect four boundary points in pairs without crossings. These are depicted in Figure~\ref{fig:basisstates}. 
	
	\begin{figure}[!h]
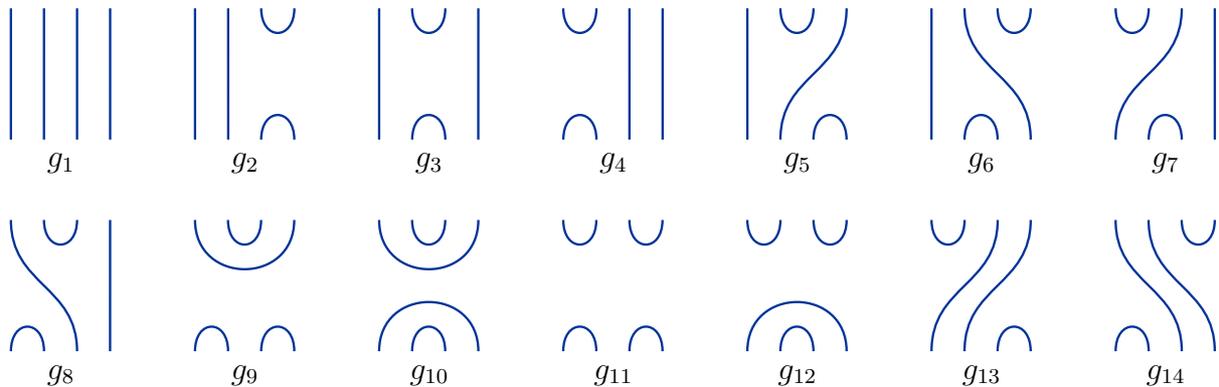

		\centering
		\makebox[0pt]{\begin{tabular}{wc{2cm}wc{2cm}wc{2cm}wc{2cm}wc{2cm}wc{2cm}wc{2cm}}
				\gi & \gii & \giii & \giv & \gv & \gvi & \gvii\\
				$g_1$ & $g_2$ &$g_3$ &$g_4$ &$g_5$ &$g_6$ &$g_7$\\[3ex]
				\gviii & \gix & \gx & \gxi & \gxii & \gxiii & \gxiv\\
				$g_8$ & $g_9$ &$g_{10}$ &$g_{11}$ &$g_{12}$ &$g_{13}$ &$g_{14}$
		\end{tabular}}
		\caption{The states basis element of $\mathcal{K}_4$}
		\label{fig:basisstates}
	\end{figure}
	
	Every state $S$ arising from a bracket expansion of $T$ can be expressed as a disjoint union
	\[S=\bigcirc^k\sqcup g,\]
	where $k\geq 0$, $ \bigcirc^k=\bigcirc\sqcup\bigcirc\sqcup\cdots\sqcup\bigcirc $ denotes the disjoint union of $ k $ circles, and $g\in\mathcal{K}_4$.
	
	Thus, by the bracket axioms \ref{itm:K1} and \ref{itm:K2}, the bracket evaluation of the state $S$  is \cite[p.~98]{Kauffman1993} \[\left<S\right>=x^k \left<g\right>.\]
	
	Consequently, the full bracket polynomial of the 4-tangle $T$ is a linear combination of the form
	
	\[
	\left<T\right>=\sum_{S} \left<S\right> = \sum_{i=1}^{14} a_i\left<g_i\right>,
	\]
	where each coefficient $ a_i\in \mathbb{Z}[x] $ counts (with powers of $x$) the number of states that reduce to the basis element $g_i$, possibly with additional circles.
	
	\begin{definition} 
		The concatenation (or multiplication) of two 4-tangles $T$ and $T'$, denoted $T T'$, is the 4-tangle formed by placing $T$  directly above $T'$  and connecting the bottom endpoints of $T$  to the corresponding top endpoints of $T'$  with non-crossing planar arcs:
		\[
		T:=\protect\includegraphics[height=1.5cm,valign=c]{tangleT},\qquad T':=\protect\includegraphics[height=1.5cm,valign=c]{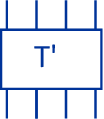},\quad  TT' = \protect\includegraphics[height=3cm,valign=c]{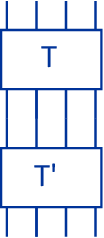}.
		\]
		
		Let $T_n:=T T\cdots T$  denote the $4$-tangle obtained by concatenating the $4$-tangle $T$  with itself $n$  times, with $T_0:=g_1$. By construction, for each $n\geq0$, the Kauffman bracket polynomial $\left<T_n\right>$  lies in the linear span of a fixed basis of tangle states $\left\{\left<g_1\right>,\ldots,\left<g_{14}\right>\right\}$. That is, there exist coefficients $a_1^{(n)},\ldots,a_{14}^{(n)}$  such that 
		\begin{equation}\label{eq:Tnformula}
			\left<T_n\right> = \sum_{i=1}^{14}a_i^{(n)}\left<g_i\right>
		\end{equation}
	\end{definition} 
	
	In the following, we compute the Kauffman bracket polynomial for $\CK_4^{2n} $ using a recursive approach based on the concatenation of a fundamental 4-tangle $G$, defined as
	\[
	G:=\go.
	\] 
	
	The following result can easily be established
	\[\begin{array}{lcl}
		\left<\go\right>&=&\left<\gi\right>+\left<\gii\right>+\left<\gii\right>+\left<\giv\right>\\[5ex] &+&\left<\gv\right>+\left<\gviii\right>+\left<\giv\right>+\left<\gvi\right>,
	\end{array}\]
	or more formally,
	\begin{equation}\label{eq:basisstates}
		\left<G\right>=\left<g_1\right>+\left<g_2\right>+\left<g_3\right>+\left<g_4\right>+\left<g_5\right>+\left<g_8\right>+\left<g_9\right>+\left<g_{11}\right>.
	\end{equation}
	
	\begin{lemma}
		For $n\geq 1$, $G_n$ satisfies the following decomposition:			
		\begingroup
		\allowdisplaybreaks 
		\begin{align*}
			\left<G_{n}\right> & = a_1^{(n-1)}\left<g_1\right> + \left(a_1^{(n-1)} + (x+2)a_2^{(n-1)} + (x+1)a_6^{(n-1)} + a_{14}^{(n-1)}\right)\left<g_2\right>\\
			& + \left(a_1^{(n-1)} + (x+1)a_3^{(n-1)} + a_5^{(n-1)} + a_{8}^{(n-1)}\right)\left<g_3\right>\\
			& + \left(a_1^{(n-1)} + (x+2)a_4^{(n-1)} + (x+1)a_7^{(n-1)} + a_{13}^{(n-1)}\right)\left<g_4\right>\\
			& + \left(a_1^{(n-1)} + (x+1)a_3^{(n-1)} + (x+2)a_{5}^{(n-1)} + a_{8}^{(n-1)}\right)\left<g_5\right>\\
			& + \left(a_2^{(n-1)} + (x+1)a_6^{(n-1)} + a_{14}^{(n-1)}\right)\left<g_6\right>\\
			& + \left(a_4^{(n-1)} + (x+1)a_7^{(n-1)} + a_{13}^{(n-1)}\right)\left<g_7\right>\stepcounter{equation}\tag{\theequation}\label{eq:Tn1}\\
			& + \left(a_1^{(n-1)} + (x+1)a_3^{(n-1)} + a_{5}^{(n-1)} + (x+2)a_{8}^{(n-1)} \right)\left<g_8\right>\\
			& + \left(a_1^{(n-1)} + (x+1)a_3^{(n-1)} + (x+2)a_{5}^{(n-1)} +(x+2)a_{8}^{(n-1)} \right.\\
			&\qquad \left. + (x^2+3x+3)a_{9}^{(n-1)} + (x^2+3x+2)a_{10}^{(n-1)}\right)\left<g_{9}\right>\\
			& + \left(a_9^{(n-1)} + (x+1)a_{10}^{(n-1)} \right)\left<g_{10}\right>\\
			& + \left(a_1^{(n-1)} + (x+2)a_2^{(n-1)} + (x+2)a_{4}^{(n-1)} +(x+1)a_{6}^{(n-1)} +(x+1)a_{7}^{(n-1)}\right.\\
			&\qquad \left.+ (x^2+3x+3)a_{11}^{(n-1)} + (x^2+3x+2)a_{12}^{(n-1)} + (x+2)a_{13}^{(n-1)} + (x+2)a_{14}^{(n-1)}\right)\left<g_{11}\right>\\
			& + \left(a_{11}^{(n-1)} + (x+1)a_{12}^{(n-1)} \right)\left<g_{12}\right>\\
			& + \left(a_{4}^{(n-1)} + (x+1)a_{7}^{(n-1)} + (x+2)a_{13}^{(n-1)} \right)\left<g_{13}\right>\\
			& + \left(a_{2}^{(n-1)} + (x+1)a_{6}^{(n-1)} + (x+2)a_{14}^{(n-1)} \right)\left<g_{14}\right>.\\
		\end{align*}
		\endgroup
	\end{lemma}
	
	\begin{proof}
		If $\left<G\right>=\sum_{j=1}^{8}\left<g_{i_j}\right>$, where the $g_{i_j}$  represent the basis states contributing to the 4-tangle $G$ as determined in \eqref{eq:basisstates}, then for the iterated 4-tangle $G_n$, we have 
		\[\left<G_n\right>=\left<G_{n-1}  G\right> = \sum_{i=1}^{14}a_i^{(n-1)}\sum_{j=1}^{8}\left<g_ig_{i_j}\right>,\]		
		where $a_i^{(n-1)}$  denotes the coefficient of the bracket $\left<g_i\right>$  in the expansion of $\left<G_{n-1}\right>$. 
		
		Each bracket $\left<g_ig_{i_j}\right>$  corresponds to the Kauffman state evaluation of the concatenated configuration formed by combining state $g_i$ from the $G_{n-1}$  factor with state $g_{i_j}$ from $G$. The result of this product is determined by a state multiplication rule which is tabulated in Table~\ref{tab:celticstatestable}. In this table, the entry in row $j$ and column $i$ gives the resulting state  from the formal product $g_ig_{i_j}$, with $i=1,\ldots,14$  indexing the column labels (top row) and $j=1,\ldots,8$  indexing the column labels (leftmost column). The conclusion follows by expressing the result in terms of the $\left<g_{i}\right>$'s.

		\begin{table}					
			\rotatebox{90}{
				\begin{minipage}{1.35\linewidth}
					\centering
					\includegraphics[width=\linewidth]{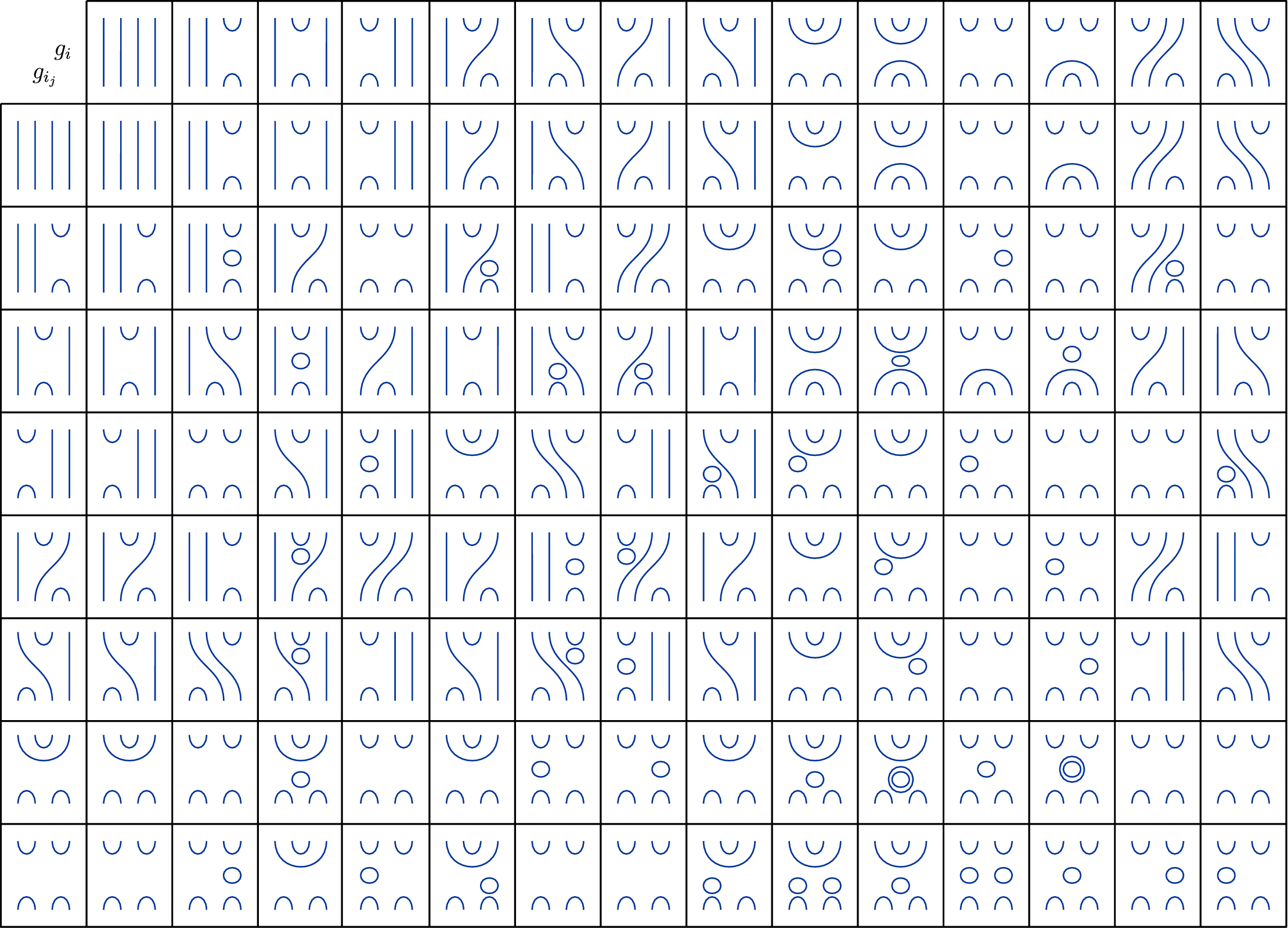}
					\caption{Multiplication table of $g_i$, $i=1,\ldots,14$ and $g_{i_j}$, $j=1,\ldots,8$}
					\label{tab:celticstatestable}
				\end{minipage}
			}					
		\end{table}				
	\end{proof}
	
	\begin{notation}
		For convenience, we identify the bracket polynomial expression in $\langle G \rangle=\sum_{i=1}^{14} a_i\left<g_i\right>$ with the coefficient vector
		\[
		A:=A(G) = \left[ a_1, a_2, a_3, a_4, a_5, a_6, a_7, a_8, a_9, a_{10}, a_{11}, a_{12}, a_{13}, a_{14} \right]^\intercal.
		\]
		Similarly, we identify the bracket polynomial $ \langle G_n \rangle $ with the vector
		\[
		A_n:=A(G_n) = \left[ a_1^{(n)}, a_2^{(n)}, a_3^{(n)}, a_4^{(n)}, a_5^{(n)}, a_6^{(n)}, a_7^{(n)}, a_8^{(n)}, a_9^{(n)}, a_{10}^{(n)}, a_{11}^{(n)}, a_{12}^{(n)}, a_{13}^{(n)}, a_{14}^{(n)} \right]^\intercal.
		\]
	\end{notation}
	
	\begin{lemma} For $n\geq 1$, the bracket vector of $G_n$ satisfies the following recursion:		
		\begin{equation}\label{eq:matrixdefinition}
			\renewcommand\arraystretch{1.325}
			\begin{pmatrix}
				a_1^{(n)}\\ 
				a_2^{(n)}\\ 
				a_3^{(n)}\\ 
				a_4^{(n)}\\ 
				a_5^{(n)}\\ 
				a_6^{(n)}\\ 
				a_7^{(n)}\\ 
				a_8^{(n)}\\ 
				a_9^{(n)}\\ 
				a_{10}^{(n)}\\ 
				a_{11}^{(n)}\\ 
				a_{12}^{(n)}\\ 
				a_{13}^{(n)}\\ 
				a_{14}^{(n)}
			\end{pmatrix}
			=\begin{pmatrix}
				1 & 0 & 0 & 0 & 0 & 0 & 0 & 0 & 0 & 0 & 0 & 0 & 0 & 0 \\
				1 & t & 0 & 0 & 0 & s & 0 & 0 & 0 & 0 & 0 & 0 & 0 & 1 \\
				1 & 0 & s & 0 & 1 & 0 & 0 & 1 & 0 & 0 & 0 & 0 & 0 & 0 \\
				1 & 0 & 0 & t & 0 & 0 & s & 0 & 0 & 0 & 0 & 0 & 1 & 0 \\
				1 & 0 & s & 0 & t & 0 & 0 & 1 & 0 & 0 & 0 & 0 & 0 & 0 \\
				0 & 1 & 0 & 0 & 0 & s & 0 & 0 & 0 & 0 & 0 & 0 & 0 & 1 \\
				0 & 0 & 0 & 1 & 0 & 0 & s & 0 & 0 & 0 & 0 & 0 & 1 & 0 \\
				1 & 0 & s & 0 & 1 & 0 & 0 & t & 0 & 0 & 0 & 0 & 0 & 0 \\
				1 & 0 & s & 0 & t & 0 & 0 & t & v & u & 0 & 0 & 0 & 0 \\
				0 & 0 & 0 & 0 & 0 & 0 & 0 & 0 & 1 & s & 0 & 0 & 0 & 0 \\
				1 & t & 0 & t & 0 & s & s & 0 & 0 & 0 & v & u & t & t \\
				0 & 0 & 0 & 0 & 0 & 0 & 0 & 0 & 0 & 0 & 1 & s & 0 & 0 \\
				0 & 0 & 0 & 1 & 0 & 0 & s & 0 & 0 & 0 & 0 & 0 & t & 0 \\
				0 & 1 & 0 & 0 & 0 & s & 0 & 0 & 0 & 0 & 0 & 0 & 0 & t
			\end{pmatrix}\begin{pmatrix}
				a_1^{(n-1)}\\ 
				a_2^{(n-1)}\\ 
				a_3^{(n-1)}\\ 
				a_4^{(n-1)}\\ 
				a_5^{(n-1)}\\ 
				a_6^{(n-1)}\\ 
				a_7^{(n-1)}\\ 
				a_8^{(n-1)}\\ 
				a_9^{(n-1)}\\ 
				a_{10}^{(n-1)}\\ 
				a_{11}^{(n-1)}\\ 
				a_{12}^{(n-1)}\\ 
				a_{13}^{(n-1)}\\ 
				a_{14}^{(n-1)}
			\end{pmatrix},
		\end{equation}
		where  
		\[
		s:= x+1, \quad t:=x+2, \quad u:=x^2 +3x +2 \quad\text{and}\quad v:=x^2 +3x +3.
		\]
	\end{lemma}
	
	\begin{proof}
		Recursion \eqref{eq:matrixdefinition} is directly derived from \eqref{eq:Tn1}.
	\end{proof}			
	
	Note that \eqref{eq:matrixdefinition} can also be written as		
	\begin{equation}
		A_{n} = M A_{n-1} = M^n A_0,
	\end{equation} with $$A_0=\left[1,0,0,0,0,0,0,0,0,0,0,0,0,0\right]^\intercal,$$ where $M$ denotes the $14\times 14$ matrix which is referred to as \textit{states matrix}.
	
	The $n$-th power of the state matrix is computed using the standard eigenvalue method. Hence, the characteristic polynomial of $M$ is given by 
	\[\chi(M,\lambda)=(\lambda-1) (\lambda-x-1)^3 \left(\lambda-x-2-r\right)^3 \left(\lambda-x-2+r\right)^3 \left(\lambda-\dfrac{p-q}{2} \right)^2\left(\lambda-\dfrac{p+q}{2} \right)^2,\]
	where \[r:=\sqrt{2x+3}, \quad p:=x^2+4x+4  \quad\text{and}\quad q:=\sqrt{x^4+4 x^3+12 x^2+20 x+12}.\]
	Then, the values for the $a_i^{n}$'s are given as follows: 
	\begingroup
	\allowdisplaybreaks
	\begin{align*}
		a_1^{(n)} &= 1\\
		a_2^{(n)} &= 2 wq x \left(x^2  -  1\right)^2 \Big\{\lambda_1^n \left(2 r x^2 - 4 r\right) + \lambda_2^n \left((r + 1) x^2 + 2 x\right) + \lambda_3^n \left((r - 1) x^2 - 2 x\right) - 4 r x^2 + 4 r\Big\}\\
		a_3^{(n)} &= 4w q x^2 \left(x^2  -  1\right)^2 \Big\{\lambda_2^n ((r + 1) x + 2) + \lambda_3^n ((r - 1) x - 2) - 2 r x\Big\}\\
		a_4^{(n)} &= 2 wq x \left(x^2  -  1\right)^2 \Big\{\lambda_1^n \left(2 r x^2 - 4 r\right) + \lambda_2^n \left((r + 1) x^2 + 2 x\right) + \lambda_3^n \left((r - 1) x^2 - 2 x\right) - 4 r x^2 + 4 r\Big\}\\
		a_5^{(n)} &= 2w q x^2 \left(x^2  -  1\right)^2 \Big\{\lambda_2^n \left(-2 x^2-2 x+2-2 r\right) + \lambda_3^n \left(2 x^2+2 x-2-2 r\right) + 4 r\Big\}\\
		a_6^{(n)} &= 2w q x^2 \left(x^2  -  1\right) \Big\{\lambda_2^n  \left(-2 x^3 - (2 r + 2) x^2 + 2 x + 2 + 2 r\right)\Big. \\
		&\qquad\Big. + \lambda_3^n  \left(2 x^3 + (2 - 2 r) x^2 - 2 x - 2 + 2 r\right)  +  4 r x^2  -  4\Big\} \\
		a_7^{(n)} &= 2w q x^2 \left(x^2  -  1\right) \Big\{\lambda_2^n  \left(-2 x^3 - (2 r + 2) x^2 + 2 x + 2 + 2 r\right)\Big. \\
		&\qquad\Big. + \lambda_3^n  \left(2 x^3 + (2 - 2 r) x^2 - 2 x - 2 + 2 r\right)  +  4 r x^2  -  4\Big\} \\
		a_8^{(n)} &= 2w q x^2 \left(x^2  -  1\right)^2 \Big\{\lambda_2^n  \left(-2 x^2 - 2 x + 2 - 2 r\right) + \lambda_3^n  \left(2 x^2 + 2 x - 2 - 2 r\right) + 4 r\Big\}\\
		a_9^{(n)} &= 2w x \left(x^2  -  1\right) \Big\{\lambda_2^n  q \left(4 x^4 + 4 x^3 + (4 r - 8) x^2 - 4 x + 4 - 4 r\right)\Big.\\
		&\qquad + \lambda_3^n  q\left(-4 x^4 - 4 x^3 + (4 r + 8) x^2 + 4 x - 4 - 4 r\right)\\
		&\qquad + \lambda_4^n  r\left(x^7 + 4  x^6 + 6  x^5 + 2  x^4  -   \left(q^2  + 8 \right)x^3 -  (2 q  + 16 )x^2 +  \left(2 q^2  - 16 \right)x + 4 q - 8\right)\\
		&\qquad + \left.\lambda_5^n  r\left(  -  x^7 - 4  x^6 - 6  x^5 - 2  x^4 +  \left(q^2  + 8 \right)x^3 +  (16  - 2 q )x^2 +  \left(16  - 2 q^2 \right)x + 4 q + 8 \right)\right.\\
		&\qquad\Big. - 4 q  x^2\Big\}\\
		a_{10}^{(n)} &= 4 wx \left(x^2  -  1\right) \Big\{ \lambda_2^n  q\left( ( - r - 1) x^3 - 2  x^2 + (r + 1) x + 2 \right)\Big.\\
		&\qquad  +  \lambda_3^n  q\left( (1 - r) x^3 + 2  x^2 +  (r - 1) x - 2 \right)\\
		&\qquad  +  \lambda_4^n  r\left(x^5 + 2  x^4 + q  x^3 - 2  x^2 - (2 q  + 4 )x - 4 \right)\\
		&\qquad  +   \Big. \lambda_5^n  r\left( -  x^5 - 2  x^4 + q  x^3 + 2  x^2 +  (4  - 2 q )x  + 4 \right) + 2 q r x\Big\}\\
		a_{11}^{(n)} &=2w x \left(x^2  -  1\right)  \Big\{ \lambda_2^n  q \left( ( - 4 r - 4) x^3 - 8  x^2 +  (4 r + 4) x + 8 \right)\Big.\\
		&\qquad + \lambda_3^n  q \left( (4 - 4 r) x^3 + 8  x^2 +  (4 r - 4) x - 8 \right)\\
		&\qquad + \lambda_4^n  r \left(  -  x^6  -  6 x^5  -  10  x^4  +  2 q  x^3   +   \left(q^2  + 20 \right)x^2  +  (24  - 4 q )x  + 8  - 2 q^2\right)\\
		&\qquad\Big.  + \lambda_5^n  r \left(   x^6  +  6 x^5  +  10  x^4  +  2 q  x^3   -   \left(q^2  + 20 \right)x^2  -  (4 q  + 24 )x  - 8  + 2 q^2\right) + 4 q r x^3\Big\}\\
		a_{12}^{(n)} &=4w x \left(x^2  -  1\right)  \Big\{ \lambda_2^n  q\left(2 x^3 + (2 r + 2) x^2 - 2 x - 2 - 2 r\right)\Big.\\
		&\qquad  + \lambda_3^n  q\left(-2 x^3 + (2 r - 2) x^2 + 2 x + 2 - 2 r\right)\\
		&\qquad  + \lambda_4^n  r\left(-  x^4 - 4  x^3 - q  x^2 + 8  x + 4  + 2 q\right)\\
		&\qquad \Big.  +  \lambda_5^n  r\left(  x^4 + 4  x^3 - q  x^2 - 8  x - 4  + 2 q\right) - 2 q r x^2\Big\}\\
		a_{13}^{(n)} &=2w q x \left(x^2  -  1\right)^2 \Big\{\lambda_1^n  \left( - 2 r x^2 + 4 r\right) + \lambda_2^n  \left((r + 1) x^2 + 2 x\right) + \lambda_3^n  \left((r - 1) x^2 - 2 x\right) - 4 r\Big\}\\
		a_{14}^{(n)} &=2 wq x \left(x^2  -  1\right)^2 \Big\{\lambda_1^n  \left( - 2 r x^2 + 4 r\right) + \lambda_2^n  \left((r + 1) x^2 + 2 x\right) + \lambda_3^n  \left((r - 1) x^2 - 2 x\right) - 4 r\Big\},
	\end{align*}
	\endgroup
	where\[
	\begin{array}{c}
		r:=\sqrt{2x+3}, \quad q:=\sqrt{x^4+4 x^3+12 x^2+20 x+12},\quad w :=\dfrac{1}{8 q r x^2 \left(x^2-2\right) \left(x^2-1\right)^2},\\ 
		\lambda_1 := x + 1, \quad
		\lambda_2 := x + 2 - \sqrt{2 x + 3}, \quad
		\lambda_3 := x + 2 + \sqrt{2 x + 3},\\
		\lambda_4 := \frac{1}{2} \left(x^2+4 x+4+\sqrt{x^4+4 x^3+12 x^2+20 x+12}\right),\\
		\text{and}\quad\lambda_5 := \frac{1}{2} \left(x^2+4 x+4+\sqrt{x^4+4 x^3+12 x^2+20 x+12}\right).
	\end{array}
	\]
	
	Next, we introduce a closure for a 4-tangle $T$, denoted $\overline{T}$, which is obtained  by connecting the two top endpoints to each other and the two bottom endpoints to each other without introducing any crossings as displayed in Figure~\ref{subfig:numeratorclosure}. In the same fashion, the closure for $G_n$, denoted $\overline{G_n}$, is illustrated in Figure~\ref{subfig:Tnclosure}.
	
	\begin{figure}[!h]
		\centering
		\begin{subfigure}[t]{0.4\textwidth}
			\centering
			\includegraphics[width=2.cm]{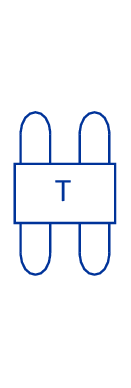}
			\caption{$\overline{T}$}
			\label{subfig:numeratorclosure}
		\end{subfigure}%
		\begin{subfigure}[t]{0.4\textwidth}
			\centering
			\includegraphics[width=2.cm]{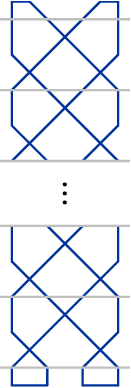}
			\caption{$\overline{G_n}$}
			\label{subfig:Tnclosure}
		\end{subfigure}%
		\caption{The closure operation}
	\end{figure}
	
	\begin{lemma}
		The bracket polynomial for the closure of $G_n$ satisfies
		\begin{equation}
			\begin{aligned}
				\left<\overline{G_n}\right> &= a_{11}^{(n)}x^4 +  \left(a_2^{(n)}+a_4^{(n)}+a_9^{(n)}+a_{12}^{(n)}+a_{13}^{(n)}+a_{14}^{(n)}\right)x^3\\
				&+\left(a_1^{(n)}+a_5^{(n)}+a_6^{(n)}+a_{7}^{(n)}+a_{8}^{(n)}+a_{10}^{(n)}\right)x^2+a_3^{(n)}x.
			\end{aligned}
		\end{equation}
	\end{lemma}
	
	\begin{proof}
		The crossings are smoothed without affecting the closure, hence 
		
		\begin{equation}\label{eq:closuredefinition}
			\left<\overline{G_n}\right>=\sum_{i=1}^{14}a_i^{(n)}\left<\overline{g_i}\right>,
		\end{equation} and the only point remaining concerns the evaluation of the brackets to the closure of the elements of $ \mathcal{K}_4 $ (see Figure~\ref{fig:Nclosure}).
		\begin{figure}[!h]
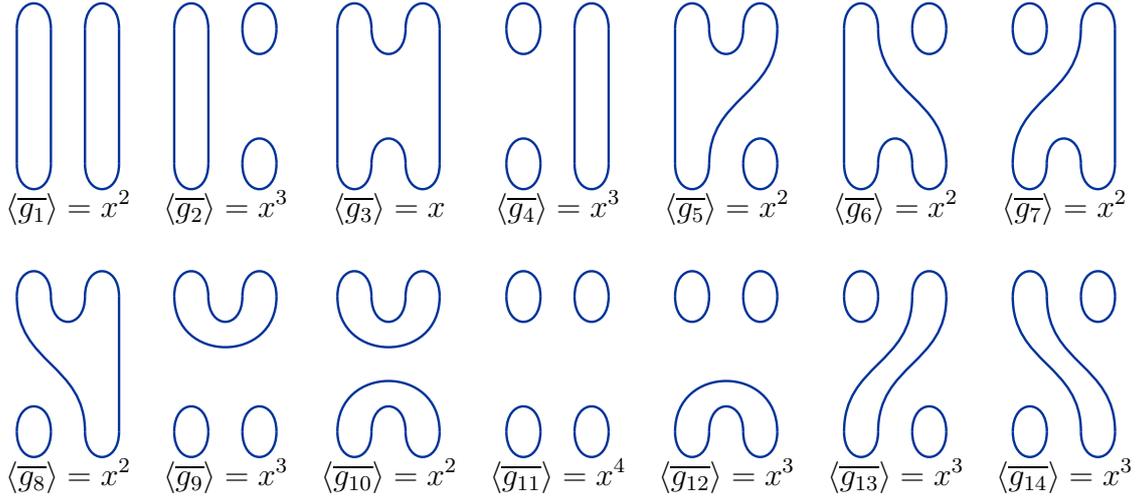

			\centering
			\makebox[0pt]{\begin{tabular}{ccccccc}
					\gicI & \giicI & \giiicI & \givcI & \gvcI & \gvicI & \gviicI \\
					$\left<\overline{g_1}\right>=x^2$ &  	$\left<\overline{g_2}\right>=x^3$   &	$\left<\overline{g_3}\right>=x$   &	$\left<\overline{g_4}\right>=x^3$   &	$\left<\overline{g_5}\right>=x^2$ & $\left<\overline{g_6}\right>=x^2$   &	$\left<\overline{g_7}\right>=x^2$\\[3ex]
					\gviiicI & \gixcI & \gxcI & \gxicI & \gxiicI & \gxiiicI & \gxivcI \\
					$\left<\overline{g_{8}}\right>=x^2$ &  	$\left<\overline{g_9}\right>=x^3$   &	$\left<\overline{g_{10}}\right>=x^2$ & $\left<\overline{g_{11}}\right>=x^4$   &	$\left<\overline{g_{12}}\right>=x^3$  &	$\left<\overline{g_{13}}\right>=x^3$   &	$\left<\overline{g_{14}}\right>=x^3$ 
			\end{tabular}}
			\caption{Closures of the elements of $\mathcal{K}_4$}
			\label{fig:Nclosure}
		\end{figure}
	\end{proof}
	
	Replacing the $a_i^{(n)}$'s and the $\left<g_i\right>$'s by their formal expressions, \eqref{eq:closuredefinition} becomes		
	\begin{align*}
		\left<\overline{G_n}\right>&= \dfrac{1}{2 q}x \left(\lambda_5^n \left(x^3+2 x^2+(q+2) x+2\right)-\lambda_4^n \left(x^3+2 x^2-(q-2) x+2\right)\right)\\
		& =\dfrac{1}{2 q}x \left(\left(\dfrac{p+q}{2}\right)^n \left(x^3+2 x^2+(q+2) x+2\right)-\left(\dfrac{p+q}{2}\right)^n \left(x^3+2 x^2-(q-2) x+2\right)\right),
	\end{align*}
	where \[p:=x^2+4x+4  \quad\text{and}\quad q:=\sqrt{x^4+4 x^3+12 x^2+20 x+12}.\]		
	
	Finally, the following Theorem establishes the link between  $\left<\overline{G_n}\right>$ and  the Kaufman bracket polynomial of the Celtic knot $\left<\CK_4^{2n}\right>$:
	\begin{theorem}\label{thm:bracketrelation}
		$\left<\overline{G_n}\right>$ and $\left<\CK^{2n}_4\right>$ satisfy the following identity:
		\begin{equation}\label{eq:NTCK}
			\left<\overline{G_n}\right> = (x + 1)^2 \left<\CK^{2n}_4\right>\quad \mbox{for $n\geq 1$}.
		\end{equation}
	\end{theorem}
	
	\begin{proof}
		Identity \eqref{eq:NTCK} follows immediately from diagram in Figure~\ref{subfig:Tnclosure}.
	\end{proof}
	
	Theorem~\ref{thm:bracketrelation} gives us the closed form for $\left<\CK^{2n}_4\right>$ based on $\left<\overline{G_n}\right>$:
	\begin{equation}\label{eq:CKfNT}
		\left<\CK^{2n}_4\right> = \dfrac{1}{(x+1)^2}\left<\overline{G_n}\right>.
	\end{equation}
	We have 
	\begingroup
	\allowdisplaybreaks
	\begin{align*}
		\left<\overline{G_n}\right>	& =\dfrac{1}{2 q}x \left(\left(\dfrac{p+q}{2}\right)^{n-1}\left(\dfrac{x^2+4x+4+q}{2}\right) \left(x^3+2 x^2+(q+2) x+2\right)\right.\\
		&\left.-\left(\dfrac{p-q}{2}\right)^{n-1}\left(\dfrac{x^2+4x+4-q}{2}\right) \left(x^3+2 x^2-(q-2) x+2\right)\right)\\ 
		& =\dfrac{1}{2 q}x \left(\left(\dfrac{p+q}{2}\right)^{n-1}\left(\dfrac{x^5}{2}+3 x^4 +(q+7) x^3+(3 q+9) x^2+\left(\dfrac{q^2}{2}+3 q+8\right) x+4+q\right)\right.\\
		&\left.-\left(\dfrac{p-q}{2}\right)^{n-1}\left(\dfrac{x^5}{2}+3 x^4 +(7-q) x^3+(9-3 q) x^2+\left(\frac{q^2}{2}-3 q+8\right)x+4-q\right)\right)\\ 
		&(\mbox{replacing $q^2$ by $x^4 + 4 x^3 + 12 x^2 + 20 x + 12$})\\
		& =\dfrac{1}{2 q}x \left(\left(\dfrac{p+q}{2}\right)^{n-1}\left(x^5+5 x^4+(q+13) x^3+(3 q+19) x^2+(3 q+14) x+4+q\right)\right.\\
		&\left.-\left(\dfrac{p-q}{2}\right)^{n-1}\left(x^5+5 x^4+(13-q) x^3+(19-3 q) x^2+(14-3 q) x+4-q\right)\right)\\ 
		& =\dfrac{1}{2 q}x \left(\left(\dfrac{p+q}{2}\right)^{n-1}(x+1)^3 \left(x^2+2 x+4+q\right)\right.\\
		&\qquad\left. -\left(\dfrac{p-q}{2}\right)^{n-1}(x+1)^3 \left(x^2+2 x+4-q\right)\right)\\
		\left<\overline{G_n}\right> &=\dfrac{1}{2q}x\left(x+1\right)^3\left(\left(x^2+2x+4+q\right)\left(\dfrac{p+q}{2}\right)^{n-1}-\left(x^2+2x+4-q\right)\left(\dfrac{p-q}{2}\right)^{n-1}\right). 
	\end{align*}
	\endgroup
	Applying \eqref{eq:CKfNT}, we obtain the expected result.

	\bigskip
	\small \textbf{2010 Mathematics Subject Classifications}:  57M25; 05A19.
\end{document}